\documentclass[11pt,reqno]{amsart}
\usepackage{CJK}
\usepackage{hyperref}
\usepackage{cite}
\usepackage{amsmath}
\usepackage{mathrsfs}
\def\dive{\operatorname{div}}

\numberwithin{equation}{section}

\newtheorem{theorem}{Theorem}[section]
\newtheorem{lemma}[theorem]{Lemma}
\newtheorem{definition}[theorem]{Definition}

\newtheorem{proposition}[theorem]{Proposition}

\newtheorem{corollary}[theorem]{Corollary}

\allowdisplaybreaks

\begin{document}
\title[\hfil Regularity of solutions to fully nonlinear elliptic equations] {Regularity of solutions to degenerate fully nonlinear elliptic equations with variable exponent}
\author[Y. Fang, V.D. R\u{a}dulescu, C. Zhang]{Yuzhou Fang, Vicen\c{t}iu D. R\u{a}dulescu$^*$ and Chao Zhang}

\thanks{$^*$Corresponding author.}

\address{Yuzhou Fang\hfill\break School of Mathematics, Harbin Institute of Technology, Harbin 150001, China}
 \email{18b912036@hit.edu.cn}

\address{Vicen\c{t}iu D. R\u{a}dulescu   \hfill\break   Faculty of Applied Mathematics,
	AGH University of Science and Technology,  Krak\'{o}w 30-059, Poland \&   Department of Mathematics, University of Craiova, Craiova 200585, Romania}
\email{radulescu@inf.ucv.ro}

\address{Chao Zhang\hfill\break School of Mathematics and Institute for Advanced Study in Mathematics, Harbin Institute of Technology, Harbin 150001, China}
 \email{czhangmath@hit.edu.cn}

\subjclass[2010]{35B65, 35J60, 35J70}   \keywords{Regularity; fully nonlinear degenerate equations; viscosity solution; double phase problem; variable exponents.}

\begin{abstract}
We consider the fully nonlinear equation with variable-exponent double phase type degeneracies
$$
\big[|Du|^{p(x)}+a(x)|Du|^{q(x)}\big]F(D^2u)=f(x).
$$
Under some appropriate assumptions, by making use of geometric tangential methods and combing a refined improvement-of-flatness approach with compactness and scaling techniques we obtain the sharp local $C^{1,\alpha}$ regularity of viscosity solutions to such equations.
\end{abstract}

\maketitle

\section{Introduction}
\label{sec0}

In this paper, we are concerned with the local regularity properties for solutions of the following fully nonlinear elliptic equation
\begin{equation}
\label{main}
\big[|Du|^{p(x)}+a(x)|Du|^{q(x)}\big]F(D^2u)=f(x)  \quad \text{in } \Omega,
\end{equation}
where $F(\cdot)$ is a uniformly $(\lambda,\Lambda)$-elliptic operator and $\Omega\subset\mathbb{R}^n$ ($n\geq2$) is a bounded domain. Equation \eqref{main} features an inhomogeneous degenerate term modelled on the double phase integrand with variable exponents
\begin{equation}
\label{0-2}
H(x,\xi)=|\xi|^{p(x)}+a(x)|\xi|^{q(x)},\quad q(x)\geq p(x)>1, a(x)\geq 0, \quad \forall x\in \Omega.
\end{equation}

From a variational point of view, \eqref{0-2} is closely related to the energy functional
\begin{equation}
\label{0-3}
u\mapsto\int(|Du|^{p(x)}+a(x)|Du|^{q(x)})\,dx,
\end{equation}
which can be regarded as a combination of the $p(x)$-growth functional and double phase possessing not only the mild phase transition but also the drastic phase transition. Such functionals were initially introduced by Zhikov in the context of homogenization (see \cite{Zhi86,Zhi95}). They provide useful models for strongly anisotropic materials whose hardening properties relevant to the exponents determining the growth order of the gradient variable switches according to the position.

These functionals with non-standard growth conditions
\begin{equation*}
u\mapsto\int_\Omega F(x,u,Du)\,dx, \quad \nu|\xi|^p\leq F(x,u,\xi)\leq L(|\xi|^q+1),
\end{equation*}
starting with the pioneering works of Marcellini \cite{Mar89,Mar91,Mar96}, have been a surge of interest over the last years. Recently, the study of nonautonomous functionals, especially for the double phase problems
$$
u\mapsto \int_\Omega (|Du|^p+a(x)|Du|^q)\,dx,
$$
has been continued in a series of remarkable papers, see for instance \cite{BCM18,CM15,CM16} and references therein. From a purely mathematical point of view, the functional \eqref{0-3}  is a natural extension of double phase functional  to the  variable exponent case. Under the assumptions that
$$
0\leq a(\cdot)\in C^{0,\gamma}, \quad 1<p(\cdot),q(\cdot)\in C^{0,\sigma},  \quad 1\leq\frac{q(x)}{p(x)}<1+\frac{\beta}{n}
$$
with $\beta:=\min\{\gamma,\sigma\}$, Ragusa and Tachikawa in \cite{RT20}  showed that the minimizer $u$ of \eqref{0-3} is locally $C^{1,\alpha}$-regular. The boundary regularity results for the variable exponent double phase problems with Dirichlet boundary conditions  were further studied by Tachikawa \cite{Tac20}. On the other hand, Byun and Lee  in \cite{BL20} considered the inhomogeneous double phase equation with variable exponents
$$
\dive(|Du|^{p(x)-2}Du+a(x)|Du|^{q(x)-2}Du)=\dive(|F|^{p(x)-2}F+a(x)|F|^{q(x)-2}F)
$$
and established the Calder\'{o}n-Zygmund type estimates
$$
(|F|^{p(x)}+a(x)|F|^{q(x)})\in L^\gamma_{\rm loc}(\Omega)\Rightarrow(|Du|^{p(x)}+a(x)|Du|^{q(x)})\in L^\gamma_{\rm loc}(\Omega), \quad \gamma>1
$$
under some suitable conditions. For more results on the double phase problems with variable exponents, we refer to \cite{BL21,MMOS19,SRRZ20} and references therein.

The non-variational counterparts of the models that have been mentioned above could be cast as degenerate fully nonlinear equations. There are an emblematic class of fully nonlinear equations with degeneracy (or singularity) exhibiting the following type
\begin{equation}
\label{0-4}
|Du|^pF(D^2u)=f.
\end{equation}
The fairly comprehensive investigation of this kind of fully nonlinear equations has been carried out. Birindelli and Demengel  \cite{BD04} established the comparison principle and Liouville-type theorems in the singular setting, and showed the regularity and uniquness of the first eigenfunction in \cite{BD10}. Alexandrov-Bakelman-Pucci estimates have also been performed for such class of equations in \cite{DFQ09,Imb11}, which were applied to deduce the Harnack inequality in \cite{DFQ10,Imb11}. In particular, Imbert and Silvestre \cite{IS13} proved the interior $C^{1,\alpha}$ estimates on the viscosity solutions to \eqref{0-4} by means of an improvement-of-flatness approach, which states that if $p\geq0$, $F(\cdot)$ is uniformly elliptic as well as the source term $f$ is bounded in $\Omega$, then, for any $U\subset\subset\Omega$, $u$ is $C^{1,\alpha}$ regular in $U$ and it holds that
$$
[u]_{1+\alpha,U}\leq C\left(\|u\|_{L^\infty(\Omega)}+\|f\|^{\frac{1}{1+p}}_{L^\infty(\Omega)}\right).
$$
Whereafter, Ara\'{u}jo, Ricarte and Teixeira in \cite{ART15} studied the optimal regularity for the viscosity solutions to \eqref{0-4}, that is, solutions are of class $C^{1,\hat{\alpha}}_{\rm loc}$ with $\hat{\alpha}=\min\{\overline{\alpha},\frac{1}{p+1}\}$. Here $\overline{\alpha}\in(0,1)$ is the H\"{o}lder index corresponding to the Krylov-Sofonov regularity for the equation $F(D^2u)=0$. It is noteworthy to mention that De Filippis in \cite{DeF21}  introduced the double phase type degeneracies to the fully nonlinear equation
$$
\big[|Du|^{p}+a(x)|Du|^{q}\big]F(D^2u)=f(x), \quad 0<p\le q,
$$
and proved the $C^{1,\gamma}$ local regularity for viscosity solutions. Moreover, the sharp local $C^{1,\gamma}$ geometric regularity estimates for bounded  viscosity solutions were obtained by da Silva and Ricarte in \cite{SR20}.  Meanwhile, under rather general conditions,  Bronzi, Pimentel, Rampasso and Teixeira in \cite{BPRT20} proved that viscosity solutions to the following variable exponent fully nonlinear elliptic equations

$$
|Du|^{p(x)}F(D^2u)=f
$$
are locally of class $C^{1,\gamma}$ for a universal constant $\gamma\in (0,1)$.  We refer the readers to \cite{BD12,BD14,BDL19,SV21} for more related results.

In this paper, motivated by the results in \cite{BPRT20, DeF21, SR20} we consider the fully nonlinear elliptic equations with variable-exponent nonhomogeneous degeneracy of the form \eqref{main}.  By making use of geometric tangential methods and combing a refined improvement-of-flatness approach with compactness and scaling techniques, we show that the viscosity solutions to \eqref{main} are locally of class $C^{1,\alpha}(\Omega)$.

\smallskip

Now we are in position to state our main result of this work.
\begin{theorem}
\label{thm0}
Suppose $u\in C(\Omega)$ is a viscosity solution to problem \eqref{main}. Under hypotheses $(A_1)$--$(A_4)$ (see Section  \ref{sec1}), there holds that $u\in C^{1,\alpha}_{\rm loc}(\Omega)$ for any $\alpha$ verifying
$$
0<\alpha<\min\left\{\overline{\alpha},\frac{1}{1+\sup_\Omega p(x)}\right\}.
$$
Moreover, we have the following H\"{o}lder estimate on the gradient of solution, that is, for any subdomain $\Omega'\subset\subset\Omega$, there is a constant $C$ depending on $n,\lambda,\Lambda,\alpha,{\rm dist}(\Omega',\Omega)$ and $\sup_\Omega p(x)$ such that
$$
\sup_{\stackrel{x,y\in\Omega'}{x\neq y}}\frac{|Du(x)-Du(y)|}{|x-y|^\alpha}\leq C\left(1+\|u\|_{L^\infty(\Omega)}+\|f\|^\frac{1}{1+\inf_\Omega p(x)}_{L^\infty(\Omega)}\right).
$$
\end{theorem}

This paper is organized as follows. In Section \ref{sec1}, we first collect some basic notations, notions and present the assumptions on equation \eqref{main}, then explain how to simplify the problem to a smallness regime. Sections \ref{sec2} and \ref{sec3} are devoted to showing the $C^{1,\alpha}$-regularity properties of solutions to \eqref{main}.

\section{Preliminaries}
\label{sec1}

In this section, we give some hypotheses on equation \eqref{main} together with definitions of Pucci operators and viscosity solutions. Additionally, we will explain how to reduce this problem to a smallness regime.

\subsection{Notions and assumptions}
%We give the definition of viscosity solution to Eq. \eqref{main} and some assumptions on this equation.
Now we give the main assumptions  as follows:

\smallskip

\begin{itemize}
\item [($A_1$)] The fully nonlinear operator $F: Sym(n)\rightarrow \mathbb{R}$ is continuous and uniformly $(\lambda,\Lambda)$-elliptic in the sense that
    $$
    \lambda\|N\|\leq F(M+N)-F(M)\leq \Lambda\|N\|
    $$
    for some $0<\lambda\leq \Lambda$ and each $M,N\in Sym(n)$ with $N\geq0$. Here $Sym(n)$ stands for the set of all $n\times n$ real symmetric matrices. Let us suppose $F(0)=0$ for convenience.

\smallskip

\item [($A_2$)] The modulating coefficient $a(x)$ satisfies that $0\leq a(\cdot)\in C(\Omega)$.

\smallskip

\item [($A_3$)] We assume that these two variable exponents $p(x),q(x)\in C(\Omega)$ such that $0\leq p(x)\leq q(x)< \infty$ for all $x\in \Omega$.

\smallskip

\item [($A_4$)] The source term $f$ belongs to $C(\Omega)\cap L^\infty(\Omega)$.
\end{itemize}

The Pucci extremal operators $P^{\pm}_{\lambda,\Lambda}:Sym(n)\rightarrow\mathbb{R}$ are defined as
$$
P^+_{\lambda,\Lambda}(M):=\Lambda\sum_{\overline{\lambda}_i>0}\overline{\lambda}_i+\lambda\sum_{\overline{\lambda}_i<0}\overline{\lambda}_i
$$
and
$$
P^-_{\lambda,\Lambda}(M):=\lambda\sum_{\overline{\lambda}_i>0}\overline{\lambda}_i+\Lambda\sum_{\overline{\lambda}_i<0}\overline{\lambda}_i,
$$
where $\{\overline{\lambda}_i\}^n_1$ are the eigenvalues of matrix $M$. With the Pucci operators in hand, the $(\lambda,\Lambda)$-ellipticity of nonlinear operator $F$ can be reformulated as
$$
P^-_{\lambda,\Lambda}(N)\leq F(M+N)-F(M)\leq  P^+_{\lambda,\Lambda}(N)
$$
for any $M,N\in Sym(n)$.

\begin{definition}
A function $u\in C(\Omega)$ is called a viscosity supersolution to \eqref{main}, if for all $x_0\in \Omega$ and $\varphi(x)\in C^2(\Omega)$ such that $u-\varphi$ attains a local minimum at $x_0$, one has
$$
\big[|D\varphi(x_0)|^{p(x_0)}+a(x_0)|D\varphi(x_0)|^{q(x_0)}\big]F(D^2\varphi(x_0))\leq f(x_0).
$$
A function $u\in C(\Omega)$ is a viscosity subsolution if for all $x_0\in \Omega$ and $\varphi(x)\in C^2(\Omega)$ such that $u-\varphi$ attains a local maximum at $x_0$, one has
$$
\big[|D\varphi(x_0)|^{p(x_0)}+a(x_0)|D\varphi(x_0)|^{q(x_0)}\big]F(D^2\varphi(x_0))\geq f(x_0).
$$
We say that $u$ is a viscosity solution to \eqref{main} if it is viscosity super- and subsolution simultaneously.
\end{definition}

\subsection{Smallness regime}

In this part, we will make use of the scaling properties of \eqref{main} to trace the problem back to a smallness regime. Namely, without loss of generality, it is possible to assume that
\begin{equation}
\label{1-1}
\|u\|_{L^\infty(B_1)}\leq 1 \quad\text{and} \quad \|f\|_{L^\infty(B_1)}\leq \varepsilon
\end{equation}
with $0<\varepsilon\ll1$. We call $u$ a normalized viscosity solution to equation \eqref{main} if hypothesis \eqref{1-1} holds. In what follows, we check its scaling features that permit us to work under assumption \eqref{1-1}. Assume that $u$ is a viscosity solution to \eqref{main}. We define $w:B_1\rightarrow \mathbb{R}$ as
$$
w(x):=\frac{u(x_0+\tau x)}{K}
$$
where $x_0\in \Omega$ and $\tau,K$ are two positive constants to be fixed later. We can check readily that $w$ solves in the viscosity sense
$$
\left[|Dw|^{\widetilde{p}(x)}+\widetilde{a}(x)|Dw|^{\widetilde{q}(x)}\right]\widetilde{F}(D^2w)=\widetilde{f}(x) \quad\text{in }  B_1.
$$
Here the notations $\widetilde{p}(x),\widetilde{q}(x),\widetilde{a}(x),\widetilde{F}(M)$ and $\widetilde{f}(x)$ are defined by
\begin{align*}
&\widetilde{p}(x)=p(x_0+\tau x),\quad \widetilde{q}(x)=q(x_0+\tau x), \\
&\widetilde{a}(x)=\left(\frac{K}{\tau}\right)^{\widetilde{q}(x)-\widetilde{p}(x)}a(x_0+\tau x), \\
&\widetilde{F}(M)=\frac{\tau^2}{K}F\left(\frac{K}{\tau^2}M\right), \\
&
\widetilde{f}(x)=\frac{\tau^{\widetilde{p}(x)+2}}{K^{\widetilde{p}(x)+1}}f(x_0+\tau x).
\end{align*}
Now set
$$
\tau:=\varepsilon^\frac{1}{2+\inf_\Omega p(x)}
$$
with $0<\varepsilon<1$ small enough, and
$$
K:=2\left(1+\|u\|_{L^\infty(\Omega)}+\|f\|^\frac{1}{1+\inf_\Omega p(x)}_{L^\infty(\Omega)}\right).
$$
It is easy to verify that the nonlinear operator $\widetilde{F}(\cdot)$ is uniformly $(\lambda,\Lambda)$-elliptic as well. By the previous definition, there holds that $0\leq\widetilde{p}(x),\widetilde{q}(x),\widetilde{a}(x)\in C(\Omega)$ and $\widetilde{p}(x)\leq\widetilde{q}(x)$. In addition, we can find that
$$
\|w\|_{L^\infty(\Omega)}\leq1 \quad\text{and}\quad \|\widetilde{f}\|_{L^\infty(\Omega)}\leq \varepsilon.
$$
Therefore, $w$ solves an equation possessing the same structure as \eqref{main} and then $w$ is in the smallness regime.

\section{$C^{1,\alpha}$-regularity of solutions}
\label{sec2}

We shall prove in this section that the viscosity solutions to \eqref{main} are locally of class $C^{1,\alpha}(\Omega)$. Now we first provide the local H\"{o}lder continuity of viscosity solutions to the following problem
\begin{equation}
\label{2-1}
\big[|Du+\xi|^{p(x)}+a(x)|Du+\xi|^{q(x)}\big]F(D^2u)=f(x)  \quad \text{in }  B_1,
\end{equation}
where $\xi$ is an arbitrary vector in $\mathbb{R}^n$. Its proof is postponed to the next section.

\begin{proposition}
\label{pro2-3}
Assume that $u\in C(B_1)$ is a normalized viscosity solution to equation \eqref{2-1}. Then we conclude that $u\in C^{0,\beta}_{\rm loc}(B_1)$ for some $\beta\in(0,1)$ and moreover it holds that
$$
[u]_{0,\beta;B_r}\leq C\left(n,\lambda,\Lambda,r,\sup_{B_1}p(x)\right)
$$
for any ball $B_r\subset\subset B_1$.
\end{proposition}

The H\"{o}lder estimate on viscosity solution provides the compactness with respect to the uniform convergence, which is the key ingredient of the following approximation result. We now recall a nice regularity result on solutions to the homogeneous problem
$$
F(D^2u)=0 \quad\text{in }  B_1
$$
with the fully nonlinear operator $F(\cdot)$ satisfying the condition $(A_1)$ above, which states that such solution is locally $C^{1,\overline{\alpha}}$ regular for $\overline{\alpha}\in (0,1)$ depending only on $n,\lambda,\Lambda$. Furthermore, there holds that, for a constant $C$ depending on $n,\lambda,\Lambda$,
$$
\|u\|_{C^{1,\overline{\alpha}}(B_{1/2})}\leq C\|u\|_{L^\infty(B_1)}.
$$
We refer to (\cite{CC95}, Chapter 5)  for more details.

%In what follows, we give a notation $p_+:=\sup_{B_1}p(x)$.

\begin{lemma}
\label{lem2-4}
Suppose that $u\in C(B_1)$ is a normalized viscosity solution to \eqref{2-1} in $B_1$. For any $\varepsilon>0$, there is $\delta>0$ that depends on $n,\lambda,\Lambda,\varepsilon$ and $\sup_{B_1}p(x)$, such that if $\|f\|_{L^\infty(B_1)}\leq \delta$, then
$$
\|u-v\|_{L^\infty(B_{1/2})}\leq \varepsilon
$$
for some $C^{1,\overline{\alpha}}(B_{3/4})$-regular function $v(x)$ with $\overline{\alpha}\in(0,1)$. Furthermore, $\|v\|_{C^{1,\overline{\alpha}}(B_{3/4})}\leq C$ with $C$ depending only on $n,\lambda,\Lambda$.
\end{lemma}

\begin{proof}
We argue by contradiction. If this claim fails, then there are $\varepsilon_0>0$ and sequences of functions $\{a_j\},\{p_j\},\{q_j\},\{f_j\},\{u_j\},\{F_j\}$ and sequence of vectors $\{\xi_j\}$ satisfying separately that

\begin{itemize}
  \item[(i)] $F_j(\cdot)$ is uniformly $(\lambda,\Lambda)$-elliptic;

    \smallskip

  \item[(ii)] $\|f_j\|_{L^\infty(B_1)}\leq \frac{1}{j}$ and $f_j\in C(B_1)$;

    \smallskip

  \item[(iii)] $0\leq p_j(x)\leq q_j(x)$, $p_j(x)\leq \sup_{x\in B_1}p(x)$ and $p_j(\cdot),q_j(\cdot)\in C(B_1)$;

    \smallskip

  \item[(iv)] $a_j(\cdot)\in C(B_1)$ and $a_j(\cdot)\geq0$ in $B_1$;

    \smallskip

  \item[(v)] $u_j\in C(B_1)$ and $\|u_j\|_{L^\infty(B_1)}\leq 1$.
\end{itemize}
Furthermore, there holds that
\begin{equation}
\label{2-12}
\big[|Du_j+\xi_j|^{p_j(x)}+a_j(x)|Du_j+\xi_j|^{q_j(x)}\big]F_j(D^2u_j)=f_j(x)  \quad \text{in } B_1.
\end{equation}
Nonetheless, for any $v(x)\in C^{1,\overline{\alpha}}(B_{3/4})$,
$$
\|u_j-v\|_{L^\infty(B_{1/2})}>\varepsilon_0.
$$

Notice that the condition (i) renders that $F_j(\cdot)$ converges uniformly to some $(\lambda,\Lambda)$-elliptic operator $\overline{F}(\cdot)\in C(Sym(n),\mathbb{R})$. Moreover, by means of condition (v) and the H\"{o}lder continuity (Proposition \ref{pro2-3}) of $u_j$, it follows that $u_j$ converges locally uniformly to some function $\overline{u}$ in $B_1$. Particularly, it holds that
$$
\overline{u}\in C(B_{3/4}) \text{ and }\|\overline{u}\|_{L^\infty(B_{3/4})}\leq1,
$$
but
\begin{equation}
\label{2-13}
\sup_{x\in B_{1/2}}|\overline{u}(x)-v(x)|>\varepsilon_0.
\end{equation}

Next, we aim at showing that $\overline{u}$ solves in the viscosity sense
\begin{equation}
\label{2-14}
\overline{F}(D^2h)=0   \quad \text{in }  B_{3/4}.
\end{equation}
To this end, we only verify that $\overline{u}$ is a viscosity supersolution, because we can prove that it is a subsolution in a similar way. We first suppose that $\varphi(x)\in C^2(B_{3/4})$ touches $\overline{u}(x)$ from below at $\overline{x}$, that is,
\begin{equation*}
\begin{cases}
\varphi(\overline{x})=\overline{u}(\overline{x}),\\
\varphi(x)<\overline{u}(x) \quad
\textmd{for } x\neq \overline{x}.
\end{cases}
\end{equation*}
Without loss of generality, we assume that $|\overline{x}|=\overline{u}(0)=0$ and $\varphi(x)$ is a quadratic polynomial, that is,
$$
\varphi(x)=\frac{1}{2}\langle Mx,x\rangle+\langle b,x\rangle.
$$
Here $\langle\cdot,\cdot\rangle$ denotes the inner product. Since $u_j\rightarrow\overline{u}$ locally uniformly in $B_1$, we find that
$$
\varphi_j(x):=\frac{1}{2}\langle M(x-x_j),(x-x_j)\rangle+\langle b,x-x_j\rangle+u_j(x_j)
$$
touches $u_j$ from below at $x_j$ lying in a small neighbourhood of origin. Due to $u_j$ a viscosity solution, we have
$$
\big[|b+\xi_j|^{p_j(x_j)}+a_j(x_j)|b+\xi_j|^{q_j(x_j)}\big]F_j(M)\leq f_j(x_j).
$$
First of all, if $\{\xi_j\}$ is unbounded, we can assume $|\xi_j|\rightarrow\infty$ (up to a subsequence), which implies that
$$
\overline{F}(M)\leq 0.
$$

In the second case, if $\{\xi_j\}$ is bounded, we may assume $\xi_j\rightarrow\overline{\xi}$ (up to a subsequence). For the case $|b+\overline{\xi}|\neq0$, it is easy to infer that $\overline{F}(M)\leq 0$ as well. Here we observe that $\{p_j(x_j)\}$ is a bounded sequence. Now let us concentrate on the case $|b+\xi|=0$ (There are two possibilities that $|b|=|\xi|=0$ or $b=-\xi$ with $|b|,|\xi|>0$). We are going to justify $\overline{F}(M)\leq 0$ in this scenario. We suppose $\overline{F}(M)>0$ by contradiction. Then we can see from ellipticity condition of $F(\cdot)$ that the matrix $M$ has at least one positive eigenvalue. Let $\mathbb{R}^n=T\oplus Q$ be the orthogonal sum. Here $T={\rm span}(e_1, e_2, \cdots, e_k)$ is the invariant space composed of those eigenvectors corresponding to positive eigenvalues.

\medskip

\textbf{Case 1.} $b=-\xi$ with $|b|,|\xi|>0$. Let $\gamma>0$ and set
$$
\phi(x):=\varphi(x)+\gamma|P_T(x)|=\frac{1}{2}\langle Mx,x\rangle+\langle b,x\rangle+\gamma|P_T(x)|,
$$
where $P_T$ stands for the orthogonal projection over $T$. Because of $u_j\rightarrow \overline{u}$ locally uniformly in $B_1$ and $\varphi(x)$ touching $\overline{u}(x)$ from below at the origin, then, for $\gamma$ small enough, $\phi(x)$ touches $u_j(x)$ from below at a point $x^\gamma_j\in B_r$ ($B_r$ is a small neighbourhood of the origin). In addition, there holds that, up to a subsequence, $x^\gamma_j\rightarrow\overline{x}$, for some $\overline{x}\in B_{3/4}$, as $j\rightarrow\infty$.

First, when $P_T(x^\gamma_j)=0$, we assert $\overline{F}(M)\leq 0$, which contradicts the previous assumption. To this aim, we note that
$$
\overline{\phi}(x):=\varphi(x)+\gamma e\cdot P_T(x)
$$
touches $u_j$ from below at $x^\gamma_j$ for every $e\in \mathbb{S}^{n-1}$ (that is, $|e|=1$). We can readily derive
$$
D\overline{\phi}(x^\gamma_j)=Mx^\gamma_j+b+\gamma P_T(e), \quad D^2\overline{\phi}(x^\gamma_j)=M.
$$
We choose $e\in T\cap \mathbb{S}^{n-1}$ so that $P_T(e)=e$. It follows from $u_j$ being viscosity solution that
\begin{equation}
\label{2-15}
\big[|\xi_j+Mx^\gamma_j+b+\gamma e|^{p_j(x^\gamma_j)}+a_j(x^\gamma_j)|\xi_j+Mx^\gamma_j+b+\gamma e|^{q_j(x^\gamma_j)}\big]F_j(M)\leq f_j(x_j).
\end{equation}
If $M\overline{x}=0$, then for $j$ sufficiently large, we arrive at
$$
|\xi_j+Mx^\gamma_j+b|\leq \frac{\gamma}{2}
$$
so that
$$
|\xi_j+Mx^\gamma_j+b+\gamma e|\geq\frac{\gamma}{2}.
$$
Furthermore, \eqref{2-15} becomes
\begin{align*}
F_j(M)&\leq\frac{f_j(x^\gamma_j)}{|\xi_j+Mx^\gamma_j+b+\gamma e|^{p_j(x^\gamma_j)}+a_j(x^\gamma_j)|\xi_j+Mx^\gamma_j+b+\gamma e|^{q_j(x^\gamma_j)}}\\
&\leq\frac{f_j(x^\gamma_j)}{|\xi_j+Mx^\gamma_j+b+\gamma e|^{p_j(x^\gamma_j)}}\\
&\leq\left(\frac{2}{\gamma}\right)^{p_j(x^\gamma_j)}f_j(x^\gamma_j),
\end{align*}
where $\{p_j(x^\gamma_j)\}_j$ is a bounded sequence. Thus we obtain $\overline{F}(M)\leq 0$ as $j\rightarrow\infty$. When $|M\overline{x}|>0$, it yields that, for $j$ large enough,
$$
|Mx^\gamma_j|>\frac{1}{2}|M\overline{x}|
$$
and
$$
|\xi_j+b|<\frac{1}{8}|M\overline{x}|.
$$
Moreover, we pick $\gamma<\frac{1}{8}|M\overline{x}|$. Then we could find
$$
|\xi_j+Mx^\gamma_j+b+\gamma e|>\frac{1}{4}|M\overline{x}|.
$$
Thereby, \eqref{2-15} turns into
$$
F_j(M)\leq\frac{f_j(x^\gamma_j)}{(4^{-1}|M\overline{x}|)^{p_j(x^\gamma_j)}}.
$$
By sending $j\rightarrow\infty$, we again have $\overline{F}(M)\leq0$. Next, we proceed with the scenario $P_T(x^\gamma_j)\neq0$. Notice that $|P_T(x)|$ is smooth and convex in a small neighbourhood of $x^\gamma_j$. Denote
$$
\zeta^\gamma_j=\frac{P_T(x^\gamma_j)}{|P_T(x^\gamma_j)|}.
$$
There holds that
$$
D(|P_T(x)|)|_{x^\gamma_j}=\zeta^\gamma_j \quad\text{and} \quad D^2(|P_T(x)|)|_{x^\gamma_j}=|P_T(x^\gamma_j)|^{-1}(I-\zeta^\gamma_j\otimes\zeta^\gamma_j).
$$
Hence through $u_j$ being a viscosity supersolution, we derive the following viscosity inequality
\begin{align*}
f_j(x^\gamma_j)\geq&
\big[|Mx^\gamma_j+b+\gamma \zeta^\gamma_j+\xi_j|^{p_j(x^\gamma_j)}+a_j(x^\gamma_j)|Mx^\gamma_j+b+\gamma \zeta^\gamma_j+\xi_j|^{q_j(x^\gamma_j)}\big]\\
&\cdot F_j(M+\gamma|P_T(x^\gamma_j)|^{-1}(I-\zeta^\gamma_j\otimes\zeta^\gamma_j)).
\end{align*}
Here we observe that $|\zeta^\gamma_j|=1$. Let
$$
e:=\zeta^\gamma_j=\frac{P_T(x^\gamma_j)}{|P_T(x^\gamma_j)|}.
$$
We could easily perform the same procedure as in the case $P_T(x^\gamma_j)=0$, via distinguishing $M\overline{x}=0$ and $M\overline{x}\neq0$. Therefore, under the condition that $b=-\xi\neq0$, we reach $\overline{F}(M)\leq0$, which contradicts the assumption $\overline{F}(M)>0$.

\medskip

\textbf{Case 2.} $b=\xi=0$. In this case, the proceedings become easier. Because $\frac{1}{2}\langle Mx,x\rangle$ touches $\overline{u}(x)$ from below at the origin and $u_j\rightarrow\overline{u}$ locally uniformly, the test function
$$
\hat{\phi}(x):=\frac{1}{2}\langle Mx,x\rangle+\gamma|P_T(x)|
$$
touches $u_j$ from below at a point $\hat{x}_j\in B_r$ for sufficiently small $\gamma>0$. Also, the sequence $\{\hat{x}_j\}$ is uniformly bounded. Likewise, we examine these two scenarios that $|P_T(\hat{x}_j)|=0$ and $|P_T(\hat{x}_j)|>0$.  The remaining work is to evaluate the boundedness on $|M\hat{x}_j+\xi_j+\gamma e|$ and $|M\hat{x}_j+\xi_j+\gamma \hat{e}|$ with $\hat{e}:=\frac{P_T(\hat{x}_j)}{|P_T(\hat{x}_j)|}$ ($P_T(\hat{x}_j)\neq0$), which is analogous to the Case 1. Eventually, we shall infer that $\overline{F}(M)\leq0$.

As has been stated above, we proves that $\overline{u}$ is a viscosity supersolution to \eqref{2-14}. In order to verify that $\overline{u}$ is a viscosity subsolution of \eqref{2-14}, it suffices to show $-\overline{u}$ is a supersolution to $\hat{F}(D^2w)=0$, where $\hat{F}(M):=-\overline{F}(-M)$ is uniformly $(\lambda,\Lambda)$-elliptic as well. Hence owing to $\overline{u}$ being a viscosity solution to \eqref{2-14}, it follows from the well-known regularity results in (\cite{CC95}, Chapter 5) that $\overline{u}\in C^{1,\overline{\alpha}}_{\rm loc}(B_{3/4})$ with some $\overline{\alpha}\in (0,1)$ and moreover $\|\overline{u}\|_{C^{1,\overline{\alpha}}(B_{1/2})}\leq C$ with $C$ depending only on $n,\lambda,\Lambda$. Thus we could choose $v:=\overline{u}$ so that a contradiction with \eqref{2-13} is reached.
\end{proof}

\begin{lemma}
\label{lem2-5}
Suppose that $u$ is a normalized viscosity solution to \eqref{2-1} . Given $\alpha\in(0,\overline{\alpha})$, there are two constants $0<\rho\ll1$ and $\delta>0$, the former depending on $n,\lambda,\Lambda,\alpha$ and the latter also depending on $\sup_{B_1}p(x)$ besides these previous parameters, such that if
$$
\|f\|_{L^\infty(B_1)}\leq \delta,
$$
then there is an affine function $l(x)=a+b\cdot x$ ($a\in \mathbb{R}, b\in\mathbb{R}^n$) fulfilling
$$
\|u-l\|_{L^\infty(B_\rho)}\leq \rho^{1+\alpha}
$$
and
$$
|a|+|b|\leq C,
$$
where $C$ only depends on $n,\lambda,\Lambda$.
\end{lemma}

\begin{proof}
For $\varepsilon>0$ to be fixed a posteriori, let $v$ be a solution to $F(D^2v)=0$, where $F(\cdot)$ is a uniformly $(\lambda,\Lambda)$-elliptic operator, which is $\varepsilon$-close to $u$ in $L^\infty(B_{1/2})$. From Lemma \ref{lem2-4}, the existence of such function $v(x)$ is ensured, provided that $\delta>0$ is small enough.

By virtue of the $C^{1,\overline{\alpha}}$-regularity of $v$, there exists a constant $C>1$ depending only on $n,\lambda,\Lambda$ such that
$$
\sup_{x\in B_\rho}|v(x)-(v(0)+Dv(0)\cdot x)|\leq C\rho^{1+\overline{\alpha}}
$$
and
$$
|v(0)|+|Dv(0)|\leq C.
$$
Set
$$
l(x)=a+b\cdot x:=v(0)+Dv(0)\cdot x.
$$
It yields that
\begin{align*}
\sup_{x\in B_\rho}|u(x)-l(x)|&\leq \sup_{x\in B_\rho}|u(x)-v(x)|+\sup_{x\in B_\rho}|v(x)-l(x)|<\varepsilon+C\rho^{1+\overline{\alpha}}.
\end{align*}
Since $0<\alpha<\overline{\alpha}$, we can choose $0<\rho\ll1$ so small that
$$
C\rho^{\overline{\alpha}-\alpha}\leq\frac{1}{2}, \quad \textmd{that is, } \rho\leq(2C)^{-\frac{1}{\overline{\alpha}-\alpha}}.
$$
Also, we fix
$$
\varepsilon=\frac{1}{2}\rho^{1+\alpha}\leq\frac{1}{2}(2C)^{-\frac{1+\alpha}{\overline{\alpha}-\alpha}}.
$$
Finally, merging these preceding displays obtains
$$
\|u-l\|_{L^\infty(B_\rho)}\leq \rho^{1+\alpha}.
$$
We complete the proof.
\end{proof}

\begin{lemma}
\label{lem2-6}
Assume that $u$ is a normalized viscosity solution to \eqref{main} in $B_1$. Given $\alpha\in (0,\overline{\alpha})\cap(0,\frac{1}{1+\sup_{B_1}p(x)}]$, there exist $0<\rho<\frac{1}{2}$ and $\delta>0$, both of which are the same as those in Lemma \ref{lem2-5}, such that if
$$
\|f\|_{L^\infty(B_1)}\leq \delta,
$$
then for each $j\in\mathbb{N}$, there is a sequence $\{l_j(x)\}$, where $l_j(x)=a_j+b_j\cdot x$ ($a_j\in\mathbb{R},b_j\in\mathbb{R}^n$), satisfying
$$
\|u-l_j\|_{L^\infty(B_{\rho^j})}\leq \rho^{j(1+\alpha)},
$$
$$
|a_j-a_{j-1}|\leq C\rho^{(j-1)(1+\alpha)}
$$
and
$$
|b_j-b_{j-1}|\leq C\rho^{(j-1)\alpha},
$$
where the constant $C$ depends only upon $n,\lambda,\Lambda$.
\end{lemma}

\begin{proof}
Arguing by induction. Obviously, this claim holds for $j=1$ by Lemma \ref{lem2-5} ($a_0=0,b_0=0$). Suppose that this conclusion holds true for $j=1,2,\cdots,k$. Now we are going to justify that for $j=k+1$. Set $u_k(x): B_1\rightarrow \mathbb{R}$
$$
u_k(x):=\frac{u(\rho^k x)-l_k(\rho^k x)}{\rho^{k(1+\alpha)}}.
$$
We can readily check that $u_k$ solves in the viscosity sense
$$
\big[|\xi_k+Du_k|^{p_k(x)}+a_k(x)|\xi_k+Du_k|^{q_k(x)}\big]F_k(D^2u_k)=f_k(x) \quad \text{in } B_1,
$$
where
\begin{align*}
&F_k(M):=\rho^{k(1-\alpha)}F(\rho^{k(\alpha-1)}M),\\
\smallskip
&f_k(x):=\rho^{k(1-\alpha(1+p_k(x)))}f(\rho^k x),\\
\smallskip
&a_k(x):=\rho^{q_k(x)-p_k(x)}a(\rho^k x)
\end{align*}
and
$$
p_k(x):=p(\rho^k x),\quad q_k(x):=q(\rho^k x), \quad \xi_k:=\rho^{-k\alpha}b_k.
$$
Obviously, $F_k(\cdot)$  is also a uniformly $(\lambda,\Lambda)$-elliptic operator. By induction, we can see that
\begin{align*}
&\|u_k\|_{L^\infty(B_1)}\leq 1, \quad \|p_k\|_{L^\infty(B_1)}\leq \sup_{B_1}p(x), \\
&\|q_k\|_{L^\infty(B_1)}\leq \sup_{B_1}q(x), \quad \|a_k\|_{C(B_1)}\leq \|a\|_{C(B_1)}.
\end{align*}
Moreover, in view of the choice of $\alpha$ we easily estimate
\begin{align*}
\|f_k(x)\|_{L^\infty(B_1)}&=\|\rho^{k(1-\alpha(1+p_k(x)))}f(\rho^k x)\|_{L^\infty(B_1)}\\
&\leq \delta\rho^{k(1-\alpha(1+\sup_{B_1}p(x)))}\leq\delta.
\end{align*}
Then the smallness assumption in Lemma \ref{lem2-5} is satisfied. Thus there is an affine function $\tilde{l}(x)=\tilde{a}+\tilde{b}\cdot x$ with $|\tilde{a}|+|\tilde{b}|\leq C(n,\lambda,\Lambda)$ such that
\begin{equation}
\label{2-16}
\|u_k(x)-\tilde{l}(x)\|_{L^\infty(B_\rho)}\leq \rho^{1+\alpha}.
\end{equation}
In the sequel, we denote
$$
l_{k+1}:=a_{k+1}+b_{k+1}\cdot x,
$$
where
$$
a_{k+1}=a_k+\rho^{k(1+\alpha)}\tilde{a} \quad\text{and}\quad b_{k+1}=b_k+\rho^{k\alpha}\tilde{b}.
$$
Thus, scaling \eqref{2-16} back, we reach that
\begin{align*}
&\|u-l_{k+1}\|_{L^\infty(B_{\rho^{k+1}})}\leq \rho^{(k+1)(1+\alpha)}, \\
&|a_{k+1}-a_k|=|\rho^{k(1+\alpha)}\tilde{a}|\leq C\rho^{k(1+\alpha)}
\end{align*}
and
\begin{align*}
|b_{k+1}-b_k|=|\rho^{k\alpha}\tilde{b}|\leq C\rho^{k\alpha}.
\end{align*}
The proof is finished.
\end{proof}

\begin{corollary}
\label{cor2-7}
Under the hypotheses of Lemma \ref{lem2-6}, we will deduce that there exists an affine function $\overline{l}(x)=\overline{a}+\overline{b}\cdot x$ with
$$
|\overline{a}|+|\overline{b}|\leq C \quad\quad (\overline{a}\in \mathbb{R},\overline{b}\in\mathbb{R}^n)
$$
such that for each $0<r\leq \rho$ with $\rho$ being  identical to that in Lemma \ref{lem2-6}
$$
\|u-\overline{l}\|_{L^\infty(B_r)}\leq Cr^{1+\alpha},
$$
where $C$ depends only on $n,\lambda,\Lambda,\alpha$.
\end{corollary}

\begin{proof}
From Lemma \ref{lem2-6}, we know that $\{a_j\},\{b_j\}$ are Cauchy sequences in $\mathbb{R}$ and in $\mathbb{R}^n$, respectively. Denote
\begin{align*}
\overline{a}=\lim_{j\rightarrow\infty}a_j, \quad \overline{b}=\lim_{j\rightarrow\infty}b_j.
\end{align*}
For any $m\geq j$, we have
\begin{align*}
|a_j-a_m|&\leq|a_j-a_{j+1}|+|a_{j+1}-a_{j+2}|+\cdots+|a_{m-1}-a_m|\\
&\leq C\rho^{j(1+\alpha)}+C\rho^{(j+1)(1+\alpha)}+\cdots+C\rho^{(m-1)(1+\alpha)}\\
&=C\rho^{j(1+\alpha)}\frac{1-\rho^{(m-j)(1+\alpha)}}{1-\rho^{1+\alpha}}.
\end{align*}
Letting $m\rightarrow\infty$, we get
$$
|a_j-\overline{a}|\leq C\frac{\rho^{j(1+\alpha)}}{1-\rho^{1+\alpha}}.
$$
Similarly,
$$
|b_j-\overline{b}|\leq C\frac{\rho^{j\alpha}}{1-\rho^\alpha}.
$$

Now fixing a $0<r\leq \rho$, we can take $j\in\mathbb{N}$ such that
$$
\rho^{j+1}<r\leq \rho^j.
$$
Furthermore,
\begin{align*}
\|u(x)-\overline{l}(x)\|_{L^\infty(B_r)}&\leq \|u(x)-\overline{l}(x)\|_{L^\infty(B_{\rho^j})}\\
&\leq\|u(x)-l_j(x)\|_{L^\infty(B_{\rho^j})}+\|l_j(x)-\overline{l}(x)\|_{L^\infty(B_{\rho^j})}\\
&\leq \rho^{j(1+\alpha)}+|a_j-\overline{a}|+\rho^j|b_j-\overline{b}|\\
&\leq \rho^{j(1+\alpha)}+\frac{C}{1-\rho^\alpha}\rho^{j(1+\alpha)}\\
&\leq \frac{1}{\rho^{1+\alpha}}\left(1+\frac{C}{1-\rho^\alpha}\right)r^{1+\alpha}.
\end{align*}
Now we complete the proof.
\end{proof}

From Corollary \ref{cor2-7}, we have known that the solution to \eqref{main}, $u$, is $C^{1,\alpha}$-regular around the origin. Then we could verify that $u$ is also $C^{1,\alpha}$-regular for every point of $B_{1/2}$ by a standard translation argument, which means that $u\in C^{1,\alpha}(B_{1/2})$. Consequently, we deduce Theorem \ref{thm0} by making use of a covering argument.

\section{H\"{o}lder continuity of solutions}
\label{sec3}

In this section, we give the proof of local H\"{o}lder estimates for viscosity solutions to \eqref{2-1}. Indeed, Proposition \ref{pro2-3} is a plain consequence of Lemmas \ref{lem2-1} and \ref{lem2-2}, which yields compactness with respect to uniform convergence. First, using Ishii-Lions method, we demonstrate that the viscosity solutions to \eqref{2-1} are Lipschitz continuous if $|\xi|$ is large enough.

\begin{lemma}
\label{lem2-1}
Suppose that $u\in C(B_1)$ is a normalized viscosity solution to \eqref{2-1}. If $|\xi|>M_0$ with $M_0>0$ depending on $n,\lambda,\Lambda,\sup_{B_1}p(x),r$, then $u\in C^{0,1}_{\rm loc}(B_1)$.
\end{lemma}

\begin{proof}
Fix $0<r<1$. We are going to show that there exist two positive constants $M_1,M_2$ satisfying
\begin{equation}
\label{2-2}
G(x_0):=\sup_{B_r(x_0)\times B_r(x_0)}(u(x)-u(y)-M_1\upsilon(|x-y|)-M_2(|x-x_0|^2+|y-x_0|^2))\leq0
\end{equation}
for each $x_0\in B_{1/2}$, where
\begin{equation*}
\upsilon(s)=
 \begin{cases}
s-\omega_0s^\frac{3}{2}  & \text{\textmd{if }} 0\leq s\leq s_0, \\
\upsilon(s_0)  & \text{\textmd{if }} s>s_0
\end{cases}
\end{equation*}
with $\omega_0>0$ such that $s_0:=\left(\frac{2}{3\omega_0}\right)^2\geq1$. Here $\omega_0$ is actually a fixed quantity, such as $\frac 13$. Thriving for contradiction. Let us assume that there is $x_0'\in B_\frac{1}{2}$ satisfying $G(x_0')>0$ for any $M_1,M_2>0$. Set
$$
\varphi(x,y):=M_1\upsilon(|x-y|)+M_2(|x-x_0'|^2+|y-x_0'|^2)
$$
and
$$
\psi(x,y):=u(x)-u(y)-\varphi(x,y).
$$
Denote by $(\hat{x},\hat{y})$ the maximum point of $\psi(x,y)$ in $\overline{B_r(x_0')}\times\overline{B_r(x_0')}$. We can choose $M_2=\left(\frac{4\sqrt{2}}{r}\right)^2$ so that
$$
|\hat{x}-x_0'|+|\hat{y}-x_0'|\leq\frac{r}{2}.
$$
Hence $(\hat{x},\hat{y})\in B_r(x_0)\times B_r(x_0)$ and obviously $\hat{x}\neq\hat{y}$.

Next, we use the Ishii-Lions lemma (\cite{CIL92}, Theorem 3.2) to obtain a limiting subject $(\hat{\xi}_1,X)$ of $u$ at $\hat{x}$ and a limiting superjet $(\hat{\xi}_2,Y)$ of $u$ at $\hat{y}$, such that the matrices $X,Y$ verify the inequality
\begin{equation}
\label{2-3}
\left(\begin{array}{cc}
X & \\[2mm]
 &-Y
\end{array}
\right)\\[2mm]
\leq \left(\begin{array}{cc}
A &-A\\[2mm]
-A &A
\end{array}\right)+(2M_2+\kappa)I
\end{equation}
with $\kappa>0$, that depends on the norm of $A$, being small enough. Here
$$
A:=M_1\left[\frac{\upsilon'(|\hat{x}-\hat{y}|)}{|\hat{x}-\hat{y}|}I+\left(\upsilon''(|\hat{x}-\hat{y}|)
-\frac{\upsilon'(|\hat{x}-\hat{y}|)}{|\hat{x}-\hat{y}|}\right)\frac{(\hat{x}-\hat{y})\otimes(\hat{x}-\hat{y})}{|\hat{x}-\hat{y}|^2}\right]
$$
and
$$
\hat{\xi}_1:=M_1\upsilon'(|\hat{x}-\hat{y}|)\frac{\hat{x}-\hat{y}}{|\hat{x}-\hat{y}|}+2M_2(\hat{x}-x'_0),
$$
$$
\hat{\xi}_2:=M_1\upsilon'(|\hat{x}-\hat{y}|)\frac{\hat{x}-\hat{y}}{|\hat{x}-\hat{y}|}-2M_2(\hat{y}-x'_0).
$$
From \eqref{2-3} we could readily deduce that all eigenvalues of $X-Y$ are below $4M_2+2\kappa$, and at least one eigenvalue of $X-Y$ is below $4M_2+2\kappa+4M_1\upsilon''(|\hat{x}-\hat{y}|)$. Note that when the number $M_1$ is large enough, then this quantity is negative. Indeed, we can select $M_1\geq\frac{4M_2+2}{3\omega_0}$. Therefore, we get
\begin{align}
\label{2-4}
P^+(X-Y)&\leq \Lambda(n-1)(4M_2+2\kappa)+\lambda(4M_2+2\kappa+4M_1\upsilon''(|\hat{x}-\hat{y}|))\nonumber\\
&=2(\lambda+\Lambda(n-1))(2M_2+\kappa)+4\lambda M_1\upsilon''(|\hat{x}-\hat{y}|).
\end{align}
In addition, we have the following inequalities in the viscosity sense
\begin{equation}
\label{2-5}
\begin{cases}
\left(|\hat{\xi}_1+\xi|^{p(\hat{x})}+a(\hat{x})|\hat{\xi}_1+\xi|^{q(\hat{x})}\right)F(X)\geq f(\hat{x}),\\
\left(|\hat{\xi}_2+\xi|^{p(\hat{y})}+a(\hat{y})|\hat{\xi}_2+\xi|^{q(\hat{y})}\right)F(Y)\geq f(\hat{y}).
\end{cases}
\end{equation}
If $|\xi|\geq M_0>0$, where $M_0$ will be determined later, then we rewrite \eqref{2-5} as
\begin{equation}
\label{2-6}
\begin{cases}
\left(||\xi|^{-1}\hat{\xi}_1+e|^{p(\hat{x})}+|\xi|^{q(\hat{x})-p(\hat{x})}a(\hat{x})||\xi|^{-1}\hat{\xi}_1+e|^{q(\hat{x})}\right)F(X)\geq |\xi|^{-p(\hat{x})}f(\hat{x}),\\
\left(||\xi|^{-1}\hat{\xi}_2+e|^{p(\hat{y})}+|\xi|^{q(\hat{y})-p(\hat{y})}a(\hat{y})||\xi|^{-1}\hat{\xi}_2+e|^{q(\hat{y})}\right)F(Y)\geq |\xi|^{-p(\hat{y})}f(\hat{y}).
\end{cases}
\end{equation}
On the other hand, we can easily estimate
$$
|\hat{\xi}_1|,|\hat{\xi}_2|\leq M_1+M_2.
$$
Now we take $M_0=4(M_1+M_2)$ so that $|\xi|^{-1}\hat{\xi}_1,|\xi|^{-1}\hat{\xi}_2\leq \frac{1}{2}$.

Merging \eqref{2-6} with the uniform ellipticity of the operator $F(\cdot)$, it yields that
\begin{align*}
&\quad\frac{|\xi|^{-p(\hat{x})}f(\hat{x})}{||\xi|^{-1}\hat{\xi}_1+e|^{p(\hat{x})}+|\xi|^{q(\hat{x})-p(\hat{x})}a(\hat{x})||\xi|^{-1}\hat{\xi}_1+e|^{q(\hat{x})}}\\
&\leq F(X)\leq F(Y)+P^+(X-Y)\\
&\leq\frac{|\xi|^{-p(\hat{y})}f(\hat{y})}{||\xi|^{-1}\hat{\xi}_2+e|^{p(\hat{y})}+|\xi|^{q(\hat{y})-p(\hat{y})}a(\hat{y})||\xi|^{-1}\hat{\xi}_2+e|^{q(\hat{y})}}
+P^+(X-Y).
\end{align*}
Utilizing \eqref{2-4} we further derive
$$
-\varepsilon\cdot 2^{\sup_{B_1}p(x)}\leq \varepsilon\cdot2^{\sup_{B_1}p(x)}+2(\lambda+\Lambda(n-1))(2M_2+\kappa)-3\lambda M_1\omega_0|\hat{x}-\hat{y}|^{-\frac{1}{2}},
$$
then
\begin{equation}
\label{2-7}
3\lambda M_1\omega_0\leq 2^{1+\sup_{B_1}p(x)}+2(\lambda+\Lambda(n-1))(2M_2+1).
\end{equation}
Thus if we choose ahead of time
\begin{equation}
\label{2-7-1}
M_1\geq\max\left\{\frac{ 2^{1+\sup_{B_1}p(x)}+2(\lambda+\Lambda(n-1))(2M_2+1)}{3\lambda\omega_0},\frac{4M_2+2}{3\omega_0}\right\}+1,
\end{equation}
we can reach a contradiction with \eqref{2-7}.

As has been stated above, we verify the claim \eqref{2-2}, which implies $u$ is Lipschitz continuous and satisfies
$$
[u]_{0,1;B_r(x_0)}\leq C\left(n,\lambda,\Lambda,\sup_{B_1}p(x),r\right).
$$
The proof is now complete.
\end{proof}

Next, we verify, in the complementary case, that the solutions to \eqref{2-1} are $\beta$-H\"{o}lder continuous in a similar way.

\begin{lemma}
\label{lem2-2}
Suppose that $u\in C(B_1)$ is a normalized viscosity solution to \eqref{2-1}. If $|\xi|\leq M_0$ with $M_0$ being the same as that in Lemma \ref{lem2-1}, then for some $\beta\in(0,1)$
$$
u\in C^{0,\beta}_{\rm loc}(B_1).
$$
\end{lemma}

\begin{proof}

The outline of this proof is similar to that in Lemma \ref{lem2-1}. At this point, it is worth observing that $\upsilon(s)=s^\beta$ with $\beta\in(\frac{1}{4},\frac{3}{4})$ and
\begin{equation}
\label{2-8}
\begin{split}
P^+(X-Y)&\leq (\lambda+\Lambda(n-1))(4M_2+2\kappa)-4\lambda M_1\beta(1-\beta)|\hat{x}-\hat{y}|^{\beta-2}\\
&\leq (\lambda+\Lambda(n-1))(4M_2+2)-4\lambda\beta(1-\beta)M_1.
\end{split}
\end{equation}
Moreover, it is easy to see
\begin{align}
\label{2-9}
|\hat{\xi}_1+\xi|&\geq M_1\beta|\hat{x}-\hat{y}|^{\beta-1}-2M_2-M_0 \nonumber\\
&=M_1(\beta|\hat{x}-\hat{y}|^{\beta-1}-4)-6M_2 \nonumber\\
&\geq M_1\left(\beta\left(\frac{r}{2}\right)^{\beta-1}-4\right)-6M_2 \nonumber\\
&\geq C_1M_2>1,
\end{align}
where the last inequality holds true if $r$ is sufficiently small. Analogously,
\begin{equation}
\label{2-10}
|\hat{\xi}_2+\xi|\geq C_2M_2>1.
\end{equation}

From \eqref{2-5} and the uniform ellipticity of $F(\cdot)$, we have
\begin{align*}
&\quad \frac{f(\hat{x})}{|\hat{\xi}_1+\xi|^{p(\hat{x})}+a(\hat{x})|\hat{\xi}_1+\xi|^{q(\hat{x})}}\\
&\leq F(X)\\
&\leq\frac{f(\hat{y})}{|\hat{\xi}_2+\xi|^{p(\hat{y})}+a(\hat{y})|\hat{\xi}_2+\xi|^{q(\hat{y})}}
+P^+(X-Y).
\end{align*}
We can further infer from \eqref{2-8},\eqref{2-9} and \eqref{2-10} that
\begin{equation*}
\frac{-\varepsilon}{(C_1M_2)^{\inf_{B_1}p(x)}}\leq \frac{\varepsilon}{(C_2M_2)^{\inf_{B_1}p(x)}}
+(\lambda+\Lambda(n-1))(4M_2+2)-4\lambda\beta(1-\beta)M_1,
\end{equation*}
after rearrangement, getting
\begin{equation}
\label{2-11}
\begin{split}
 \frac{1}{4}\lambda M_1&\leq4\lambda\beta(1-\beta)M_1\leq (\lambda+\Lambda(n-1))(4M_2+2)+\frac{2}{(CM_2)^{\inf_{B_1}p(x)}}\\
 &\leq(\lambda+\Lambda(n-1))(4M_2+2)+2.
 \end{split}
\end{equation}
Thus from the selection of $M_1$ in \eqref{2-7-1}, then we will derive a contradiction with \eqref{2-11}.

Thereby, if $|\xi|\leq M_0$, then $u$ is $\beta$-H\"{o}lder continuous with the estimate
$$
[u]_{0,\beta;B_r(x_0)}\leq C\left(n,\lambda,\Lambda,r,\sup_{B_1}p(x)\right).
$$
We now finish the proof.
\end{proof}

\section*{Acknowledgements}
This work was supported by the National Natural Science Foundation of China (Nos. 12071098, 11871134). The work of Vicen\c tiu D.~R\u adulescu was also supported by a grant of the Romanian Ministry of Education and Research, CNCS-UEFISCDI, project number PN-III-P4-ID-PCE-2020-0068, within PNCDI III.

\end{document}